\documentclass[12pt]{amsart}

\usepackage{amsmath}
\usepackage{amssymb}
\usepackage{amsthm}
\usepackage{amsfonts}
\usepackage{graphicx}
\usepackage{color}
\usepackage{mathrsfs}
\usepackage{tikz}
\usepackage{booktabs}
\usepackage{hyperref}

\newtheorem{thm}{Theorem}[section]

\newtheorem{lem}[thm]{Lemma}
\newtheorem{cor}[thm]{Corollary}

\renewcommand{\a}{\alpha}

\renewcommand{\d}{\delta}
\providecommand{\e}{\varepsilon}

\providecommand{\s}{\sigma}

\newcommand{\N}{\mathbb N}

\newcommand{\sL}{\mathscr{L}}
\newcommand{\sF}{\mathscr{F}}

\DeclareMathOperator{\join}{Join}
\DeclareMathOperator{\pad}{Pad}
\newcommand{\Mer}[1]{\langle #1 \rangle_{M}}
\newcommand{\Mu}[1]{\langle #1 \rangle_{\mu}}

\newcommand{\proj}{P}

\title[Optimal Pebbling]{Optimal Pebbling of Complete Binary Trees and a
Meta-Fibonacci Sequence}
\author{Thomas M.\ Lewis \\ Fabian Salinas}

\begin{document}
\bibliographystyle{amsplain} 
\maketitle

\section{Introduction}\label{intro}
\subsection{Prologue}
In 1999, H.\ Fu and C.\ Shiue \cite{MR1771391} introduced
an algorithm to produce an optimal pebbling and the 
optimal pebbling number of a complete binary tree. 
In this paper, we present an alternative approach to 
this problem; our analysis reveals a curious connection
between the optimal pebbling number of a complete
binary tree and the Conolly-Fox sequence, 
a type of meta-Fibonacci sequence.

Let $G$ be a graph with vertex set $V$ and edge set $E$. A 
\emph{pebbling configuration} on $G$ is a function
$f : V \to \N \cup \{0\}$. For $v \in V$, we think
of $f(v)$ as the number of \emph{pebbles} at $v$
and $f(G) = \sum_{w \in V} f(w)$ as the number of pebbles
on $G$. For each positive integer $p$, let $\sF_p (G)$ denote the collection 
of pebbling configurations of $G$ containing $p$ pebbles. 

A \emph{pebbling move} on $G$ consists of removing two pebbles
from a vertex and placing a single pebble at an adjacent vertex.
In effect, to move a pebble, we must pay a pebble.
A configuration $f$ is said to \emph{pebble} $G$ provided that 
given any vertex $v$, there exists a sequence of pebbling moves
(possibly empty) that brings a pebble to $v$.

There are two numbers frequently associated with pebblings 
of a graph $G$: the \emph{pebbling number} of $G$ is
\[
\pi (G) = \min \{p : \text{$\forall f \in \sF_p (G)$, $f$ pebbles $G$}  \},
\]
and the \emph{optimal pebbling number} of $G$ is
\[
\pi^* (G) = \min \{p : \text{$\exists f \in \sF_p (G)$ such that $f$ pebbles $G$}  \}.
\]
A pebbling configuration $f \in \sF_p (G)$ for which $f(G) = \pi^* (G)$
is called an \emph{optimal pebbling} of $G$.
In this paper, we produce the optimal pebbling numbers 
and optimal pebbling configurations of complete binary trees.

A \emph{tree} is an undirected graph 
in which any two vertices are connected by exactly one path.
We adopt the following terminology regarding trees.
A \emph{rooted tree} is a tree with a distinguished
vertex, called the \emph{root} of the tree.
Let $T$ be a rooted tree with vertex set $V$, edge set $E$,
and root $r$. Given a vertex $v \in V$, the distance 
from $r$ to $v$, denoted by $d(r,v)$, 
is the number of edges in the path from $r$ to $v$. 
For each nonnegative integer $k$, 
the $k$th \emph{level} of $T$ is the set of vertices
$L_k = \{v \in V : d(r,v) = k \}.$
A \emph{leaf} of $T$ is a vertex with degree one.
The \emph{children} of a non-leaf vertex $v$ are the vertices
in the next highest level that are adjacent to $v$.
A \emph{complete binary tree} is a rooted
tree in which each non-leaf vertex has two children
and every leaf vertex is in the same level.\footnote{This 
definition of a complete binary tree is not
universally recognized. 
Our useage follows Cormen, et al.\ \cite{MR1848805}
and, notably, Fu and Shiue \cite{MR1771391}.
}
The \emph{height} of a complete binary tree is
the distance from the root to any leaf. Hereafter
$T^h$ denotes a complete binary tree of
height $h$.

To describe our results,
we first introduce a sequence of partial sums.
Given a list of lists $\sL_1, \ldots, \sL_i$, let
$\join[\sL_1,\ldots, \sL_i]$ be the list obtained by 
concatenating (in order) $\sL_1$ through $\sL_i$.
We define a list of numbers composed entirely of 1s and 5s.
We begin with $A_1 = \{5\}$. We define successive lists recursively:
for each $k \ge 2$, let
\[
A_k = \join[A_{k-1}, A_{k-1}, \{1\}].
\] 
Some examples of the lists $\{A_k\}$ are collected in Table \ref{ASequence}.

\begin{table}[htb]
\caption{The lists $A_1$ through $A_4$}
\begin{tabular}{r|l}\toprule
$n$	& $A_n$	\\	\hline
1	& $\{5\}$												\\
2	& $\{5, 5, 1\}$										\\
3	& $\{5, 5, 1, 5, 5, 1, 1\}$							\\
4	& $\{5, 5, 1, 5, 5, 1, 1, 5, 5, 1, 5, 5, 1, 1, 1\}$		\\ 
\bottomrule
\end{tabular}
\label{ASequence}
\end{table}

Let $A$ be the limit of this sequence of lists. 
Let $a_0=0$ and, for $n \ge 1$, let $a_n$ denote the $n$th element
of the list $A$. Let $\{s_n\}$ or $\{s(n)\}$ denote the sequence of partial sums
of $\{a_n\}$; see Table \ref{PartialSums}.

\begin{table}[htb]
\caption{Some terms of $\{a_n\}$ and $\{s_n\}$}
\begin{tabular}{rrrrrrrrrrrrrrrrr} \toprule
$n$		&0 	&1	&2	&3	&4	&5	&6	&7	&8	&9	&10	&11	&12	&13	&14	&15	 \\ \hline
$a_n$	&0	&5	&5 	&1 	&5	&5 	&1 	&1 	&5	&5 	&1 	&5	&5 	&1 	&1  &1 		\\ \hline
$s_n$ 	&0	&5	&10	&11	&16	&21	&22	&23	&28	&33	&34	&39	&44	&45	&46	&47		\\ 
\bottomrule
\end{tabular} 
\label{PartialSums}
\end{table}

\subsection{Results}
For each positive integer $h$, let
\[
k(h) = \max \{ k : s_k \le 2^h \}.
\]
The main result of our paper is that
$\pi^* (T^h) = 2^h - k(h)$. 
For example, since $s_8 \le 2^5$
and $s_9 > 2^5$, it follows that $k(5)=8$ and $\pi^* (T^5) = 24$.
Our presentation of this main result is divided into two parts, 
corresponding to upper and lower bounds for $\pi^* (T^h)$.
In Theorem \ref{MainUpperResult}, we assert that
$\pi^* (T^h) \le 2^h - k(h)$; the proof of this theorem, which 
can be found in \S\ref{upper}, consists of constructing an explicit 
pebbling configuration $f$ on $T^h$
satisfying $f(T^h) = 2^h - k(h)$.
In Theorem \ref{MainLowerResult}, we assert that 
$\pi^* (T^h) \ge 2^h - k(h)$; in the proof of this theorem,
which can be found in \S\ref{lower}, we show, by way of a simple necessary condition, that no configuration of fewer than $2^h - k(h)$ pebbles can pebble $T^h$.

In \S\ref{asymptotics}, we present an asymptotic expansion of $k(h)$ that
refines the work of Fu and Shiue in the concluding remarks of their paper.
We show that, as $h \to \infty$,
\begin{equation}\label{kExpansion}
k(h) = \frac{1}{3}(2^h) - \frac{1}{3}(h+2) 
	- \frac{1}{3} \alpha (h) \log_2 (h+2) + O(1).
\end{equation}
The function $\a (h)$, which appears in the third-order term
of the expansion, is bounded between $-1$ and $+1$ and satisfies 
$\liminf \a(h) = -1$ and $\limsup \a(h) = 1$. 
In effect, we show that the tree $T^h$ can be split into top and 
bottom layers. The bottom layer begins at a level 
asymptotic to $\log_2 (h/2+1)$. We demonstrate that 
our optimal pebbling configurations follow a strict pattern on the bottoms
of the trees but vary chaotically as a function of $h$ on their tops.
For example, some optimal configurations have empty tops, while others contain
pebbles at every level. The oscillatory nature of $\a (h)$ 
bears the imprint of this chaotic behavior.

Finally, in \S\ref{metafib}, 
we show that the sequence $\{s_n\}$ is related to the Conolly-Fox sequence,
a type of meta-Fibonacci sequence. The first two terms of the Conolly-Fox sequence are $c(1)= 1$ and $c(2)=2$.
Thereafter, for $n \ge 2$, 
the sequence satisfies the nested recurrence relation
\begin{equation}\label{ConollyFox}
c(n) = c(n - c(n-1)) + c(n-1 - c(n-2)).
\end{equation}
The Conolly-Fox sequence
is a variation of the Conolly sequence
(\href{https://oeis.org/A046699}{OEIS, A046699}) \cite{OEIS}.
The Conolly sequence has a different set of initial conditions,
but Nathan Fox has argued persuasively that the initial conditions
$c(1)=1$ and $c(2)=2$ are more natural \cite{Fox2019}.
We show that $s_n = 4 c_n + n.$

\subsection{Background and related work}
The paper of Fu and Shiue \cite{MR1771391} concerns
the optimal pebbling number of a complete $m$-ary tree,
and our results lean heavily on their work.
For an integer $m$, $m \ge 2$, a complete $m$-ary tree,
is a rooted tree in which each non-leaf vertex has $m$ children
and every leaf vertex is in the same level. 
In their work, an $m$-ary tree of height $h$ is denoted
by $T_m^h$. For $m \ge 3$, they show that $\pi^* (T^m_h) = 2^h$;
an optimal pebbling of $T_m^h$ consists of placing all of the
pebbles at the root.
For $m = 2$, among other things, 
Fu and Shiue produce an optimal pebbling configuration
of $T_2^h$ through an integer linear programming algorithm called OPCBT.

At this point, let us make a cursory comparison of our methods. 
Roughly speaking, OPCBT is a bottom-up algorithm: an optimal pebbling
configuration of $T^h$ (or $T_2^h$) is revealed in successive 
steps, starting from the leaves and terminating at the root. One aspect of this approach
is that $\pi^* (T^h)$ is not known until the algorithm terminates
and the configuration is examined.
By comparison, our method demonstrates that $\pi^* (T^h) = 2^h - k(h)$,
and, as we will see, an optimal configuration of $T^h$ is easily calculated 
from $k(h)$. It should be noted that these methods do not necessarily 
produce the same optimal configuration. For example,
the optimal configuration of $T^7$ produced by OPCBT 
places two pebbles at each of the nodes in levels 2, 3, and 5, and zero pebbles at
every remaining node. By comparison, an optimal configuration of $T^7$ produced by our method
places four pebbles at each node in level 1, two pebbles at each node 
in levels 3 and 5, and zero pebbles at every remaining node. 

While graph pebbling grew out of problems in 
combinatorial number theory and group theory, it was formally
introduced in its present form by Chung \cite{MR1018531} in her analysis
of the pebbling number of the hypercube. 
The optimal pebbling number of a graph was introduced later
by Pachter, Snevily, and Voxman \cite{MR1369255}.
The paper of Hurlbert \cite{MR3030615} is an excellent 
survey of graph pebbling.

The optimal pebbling numbers of some classes of graphs 
have been studied. For example, the optimal pebbling numbers 
have been determined
for caterpillars \cite{MR2499997}, 
the squares of paths and cycles \cite{MR3203678}, 
spindle graphs \cite{MR4031681},
staircase graphs \cite{MR3957918}, and 
grid graphs \cite{MR3489733, MR4090527}.
The optimal pebbling numbers have been studied for
products of graphs \cite{MR2829275}, 
graphs with a given diameter \cite{MR3991624},
and graphs with a given minimum degree \cite{MR3944614}.

In recent years, a variety of adaptations and analogs of optimal pebbling 
have emerged. For example, the optimal pebbling number of a graph 
has been extended in a variety of ways by restricting the capacity of the 
pebbling configuration or by placing 
additional requirements on the pebbling configuration; see, for example,
\cite{MR3612586, MR3956289, MR3944631, MR4092629}. 
Graph rubbling is a cognate of pebbling; readers who are interested
in graph rubbling should consult \cite{MR3944610, MR3900331, MR2528207, MR4114886}.

\section{Some essential lemmas}
\label{lemmas}

In this section, we present some essential properties of the 
sequence $\{s_n\}$. We begin by describing the $M$-expansion 
and $\mu$-expansion of a positive integer.

For each positive integer $i$, 
let $M_i = 2^i-1$ denote the $i$th Mersenne number.
Given a positive integer $n$, let 
$\ell = \max \{i : n \ge M_i\}$ and write $n = M_\ell + r$, where $0 \le r \le M_{\ell}$. If $r=0$,
then we stop and write $n = M_{\ell}$. If $r=M_{\ell}$, then we stop
and write $n=2 M_{\ell}$. If else, then we continue this process 
with $r$. In this way, we can write
\begin{equation}\label{M-expansion}
n = \e_1 M_1 +  \cdots + \e_\ell M_\ell,
\end{equation}
where $\e_i \in \{0,1,2\}$ for each $i \in [\ell]$, and
if $\e_j = 2$ for some $j \in [\ell]$, then $\e_i=0$ for 
all $i \in [j-1]$. We call the sum on the right side of equation (\ref{M-expansion})
the \emph{$M$-expansion} of $n$.
Let $\Mer{n} = \{\e_1, \ldots, \e_\ell\}$ denote the
coefficient list of the $M$-expansion of 
$n$; the entries of $\Mer{n}$ are called the \emph{$M$-digits} of $n$.
For example, $\Mer{47}=\{1,0,0,1,1\}$ and $\Mer{157} = \{0,0,0,2,0,0,1\}$.

The $\mu$-expansion of a positive integer $n$ is developed in parallel fashion.
For each positive integer $i$, let $\mu_i = 3M_i + 2 = 2^{i+1}+2^i - 1$.
If $n \le 4$, then we stop. If else, $n \ge 5$ and we let 
$\ell = \max \{i : n \ge \mu_i \}$ and write $n = \mu_\ell + r$, 
where $0 \le r \le \mu_{\ell}$. 
If $r=0$,
then we stop and write $n = \mu_{\ell}$. If $r=\mu_{\ell}$, then we stop
and write $n=2 \mu_{\ell}$. If else, then we continue this process 
with $r$. In this way, we can write
\begin{equation}\label{Mu-expansion}
n = r + \e_1 \mu_1 +  \cdots + \e_\ell \mu_\ell,
\end{equation}
where $r \in \{0,1,2,3,4\}$, $\e_i \in \{0,1,2\}$ for each $i \in [\ell]$,
and if $\e_j = 2$ for some $j \in [\ell]$, then $r=0$ and $\e_i=0$ for 
all $i \in [j-1]$. We call the sum on the right side of equation (\ref{Mu-expansion})
the \emph{$\mu$-expansion} of $n$. When $r=0$, 
we let $\Mu{n} = \{\e_1, \ldots, \e_\ell\}$ denote the
coefficient list of the $\mu$-expansion of 
$n$; the entries of $\Mu{n}$ are called the \emph{$\mu$-digits} of $n$.
For example, $409 = 3 + \mu_3 + \mu_7$, $140=2\mu_2 + \mu_3 + \mu_5$,
and $\Mu{140} = \{0,2,1,0,1\}$.

We use the following notation when working with lists.
Let $\ell$ be a positive integer and let $\sL = \{a_1, \ldots, a_{\ell}\}$
be a list of real numbers. The \emph{length} of $\sL$ is $\ell$. 
For $k \in [\ell]$, let $\proj_k (\sL) = a_k$. 
Let $\s(\sL) = a_1 + \cdots + a_{\ell}$ and, for $\ell >1$,
let $S(\sL) = \{a_2, \ldots, a_{\ell}\}$. In other words, $\s(\sL)$ 
is the sum of the entries of  $\sL$, and $S(\sL)$ is the left-shift of 
 $\sL$.
For example, since $\Mer{47}=\{1,0,0,1,1\}$, it follows that
$\proj_1 (\Mer{47}) = 1$, $\s(\Mer{47}) = 3$, and  $S(\Mer{47}) = \{0,0,1,1\}$.

\begin{lem}\label{ExpandS}
For each positive integer $n$,  
$\Mu{s_n} = \Mer{n}$.
\end{lem}

\begin{proof} 
We begin by proving a provisional form of this theorem; namely, 
for each positive integer $k$, $s_{M_{k}} = \mu_k$. 
This is true for $k=1$ by inspection: $s_{M_1} = s_1 = 5 = \mu_1$.
Let $k$ be a positive integer. Recall that the list $A_{k+1}$ 
contains $M_{k+1}$ terms and has the form
\begin{equation}
\label{Alist}
A_{k+1} = \join [A_k, A_k, \{1\}].
\end{equation}
Thus $s_{M_{k+1}} = 2 s_{M_k} + 1$. By induction, it follows that
$s_{M_k} = \mu_k$.

Let $n$ be a positive integer and let $\ell = \max \{i : n \ge M_i\}$
Then $n = M_\ell + r$, where 
$0 \le r \le M_\ell.$ Referring once again to equation (\ref{Alist}),
we see that $s_{M_\ell+r}$ is the sum of the first $M_\ell+r$ terms in
 $A_{\ell+1}$, read left to right. Clearly this is the sum of the terms of  $A_\ell$ plus the first $r$ terms of list $A_\ell$, that is,
$s_n = s_{M_\ell} + s_r$. If $r=0$, then $n = M_\ell$ and $s_n = \mu_{\ell}$,
and if $r = M_\ell$, then $n = 2 M_{\ell}$ and $s_n = 2 \mu_{\ell}$.
In either case we are done. Otherwise $0 < r < M_\ell$ and 
we continue by developing the $\mu$-expansion of $s_r$ as above.
\end{proof}

\begin{lem}\label{SFormula}
For each positive integer $n$, 
$s_n = 3 n + 2 \s(\Mer{n}).$
\end{lem}

\begin{proof} In accord with Lemma \ref{ExpandS},
let $\Mu{s_n} = \Mer{n} = \{\e_1, \ldots, \e_\ell \}$.
Since $\mu_k = 3M_k + 2$ for each positive integer $k$, it follows that
\[
s_n = 3 (\e_1 M_1 + \cdots + \e_\ell M_\ell ) + 2 (\e_1 + \cdots + \e_\ell ) 
= 3 n + 2 \s(\Mer{n}),
\]
as was to be shown.
\end{proof}

Given an integer $n > 2$, the \emph{reduction} of $n$ is the
integer $r(n)$ satisfying $\Mer{r(n)}= S(\Mer{n})$. For example,
$r(40)=18$ since $\Mer{40} = \{2,0,1,0,1\}$ and $\Mer{18} =\{0,1,0,1\}$.
For completeness, we define $r(1) = r(2) = 0$.

\begin{lem}\label{SReductionFormula}
For each positive integer $k$, 
\[
\frac{s_k - \s(\Mer{k})}{2} - 2 \proj_1 (k)  = s_{r(k)}.
\]
\end{lem}

\begin{proof}
Let $k$ be a positive integer and let $\Mer{k} = \{\e_1, \ldots, \e_\ell\}$.
Then $\s(\Mer{k}) = \e_1 + \cdots + \e_\ell$ and, by Lemma \ref{ExpandS}, $s_k = \e_1 \mu_1 + \cdots + \e_\ell \mu_\ell$. Thus 		
\[
(s_k - \s(\Mer{k}))/2 = 
	\e_1 (\mu_1-1)/2 + \e_2 (\mu_2-1)/2 + 
	\cdots + \e_\ell (\mu_\ell-1)/2.
\]
But $(\mu_1-1)/2 = 2$ and, for $j \in \{2,\ldots,\ell\}$,
$(\mu_j-1)/2 = \mu_{j-1}$; thus,
\[
(s_k - \s(\Mer{k}))/2 = 2 \e_1  + \e_2 \mu_1 + \cdots  
	+ \e_{\ell} \mu_{\ell-1}.
\]
Finally, since $\e_1 = \proj_1(k)$, we obtain
\[
\frac{s_k - \s(\Mer{k})}{2} - 2 \proj_1 (k) = 
	\e_2 \mu_1 + \cdots + \e_\ell \mu_{\ell-1} = s_{r(k)},
\]
as was to be shown.
\end{proof}

This corollary is a trivial but useful consequence
of Lemma \ref{SReductionFormula}.

\begin{cor}\label{ReduceUpperBound}
For each positive integer $k$, $s_k \ge 2 s_{r(k)}$.
\end{cor}

We close this section with the following lemma.

\begin{lem}\label{UpperBound}
Let $h$ and $k$ be a positive integers. 
If $s_k \le 2^h$, then $k < M_{h-1}$.
\end{lem}

\begin{proof} We prove the inverse. Since $s_{M_{h-1}} = \mu_{h-1} >  2^h$
and since $\{s_k\}$ is strictly increasing, it follows that
if $k \ge M_{h-1}$, then $s_k > 2^h$.\end{proof}

\section{The upper bound of $\pi^* (T^h)$}
\label{upper}

In this section, we give an explicit construction of 
a pebbling configuration $f$ on $T^h$ such that 
$f(T^h) = 2^h - k(h)$; see Theorem \ref{MainUpperResult}.

A pebbling configuration $f$ on $T^h$ is called \emph{symmetric}
provided that $f(v) = f(w)$ whenever the vertices $v$ and $w$ are in the same level. When $f$ is a symmetric pebbling of $T^h$, we write
$f = \{ f_0, f_1, \ldots, f_h\}$,
where $f_i$ is the number of pebbles at level $i$.
The number of pebbles at the root, $f_0$, is called the 
\emph{head} of $f$. In this section we consider
only symmetric pebbling configurations of $T^h$.
A pebbling configuration $f$ on $T^h$ is called \emph{even} 
provided that $f(v)$ is even for any non-root vertex $v$.

We think of the tree $T^h$ with root $\rho$
as composed of a left and right sub-tree; these
sub-trees are isomorphic to $T^{h-1}$ with roots 
labeled $\rho_L$ and $\rho_R$;
see Figure \ref{RecursiveTree}.

\begin{figure}[htb]
\includegraphics{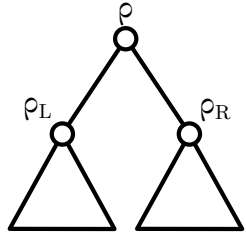}
\caption{The tree $T^h$ with root $\rho$; the left and right
sub-trees of $T^h$ have roots $\rho_L$ and $\rho_R$.}
\label{RecursiveTree}
\end{figure}

Let $f=\{f_0, \ldots, f_h\}$ be an even pebbling configuration on $T^h$.
Let $n(f) = f(T^h) = \sum_{i=0}^h 2^i f_i$ denote the number of pebbles
on $T^h$ and let $S(f) = \{f_1, f_2, \ldots, f_h\}$, the left-shift
of $f$. We can think of $S(f)$ as the pebbling configuration induced
by $f$ onto either the right or the left sub-trees of $T^h$.
Let $c(f) = \sum_{i=0}^h f_i$; this is 
the largest number of pebbles that can be amassed by $f$
at the root, $\rho$.

The \emph{reduction} of $f$, denoted by $r(f)$, is the
pebbling configuration on the left sub-tree obtained by 
transporting the maximum number of pebbles from the right sub-tree 
and the the root of $T^h$ onto the root of the left sub-tree, $\rho_L$.
Thus $r(f)= \{f_0', f_2, \cdots, f_h\},$
where
\begin{equation}\label{Head}
f_0' = f_1 + \left\lfloor \frac{f_0 + c(S(f))/2}{2} \right\rfloor
\end{equation}
For example, if $f=\{4,2,0,2,0,0\}$, then 
$r(f) = \{5,0,2,0,0\}$. 

\begin{lem}\label{Reduction}
A configuration $f$ pebbles $T^h$ if and only if at least one pebble can 
be brought to the root of $T^h$ and $r(f)$ pebbles $T^{h-1}$.
\end{lem}

Let $h$ be a positive integer; we define a family
of pebbling configurations on the tree $T^h$. 
Given a list $L$ of length $\ell$, $\ell \le h$, 
let $\pad_h (L)$ be the list of length $h$ obtained by 
padding $L$ on the right by 0's.
Given a nonnegative integer $k$ such that $s_k \le 2^h$,
let
\[
f_{h,k} = \join \left[ \{ 2^h-s_k\}, 2 \pad_h(\Mer{k}) \right].
\]
A few remarks on this definition are in order.
Since $s_k \le 2^h$, the head of $f_{h,k}$ is nonnegative.
Furthermore, by Lemma \ref{UpperBound}, the length of $\Mer{k}$
is less than or equal to $h-2$. Thus  
$f_{h,k}$ is a well-defined pebbling configuration 
on $T^h$. For $h=5$ and $k=8$, we observe that 
$s_{8} = 28$ and $\Mer{8} = \{1,0,1\}.$
Therefore, the head of $f_{5,8}$ is $2^5-28=4$ and 
$2 \pad_5(\Mer{8}) = \{2,0,2,0,0\}$. Consequently,
$f_{5,8} 
= \{4,2,0,2,0,0\}.$ 

\begin{lem}\label{NumF}
Let $h$ be a positive integer and let $k$ be a nonnegative integer
satisfying $s_k \le 2^h$. Then $n(f_{h,k}) = 2^h - k$.
\end{lem}

\begin{proof} The result is clear by inspection for $k=0$. 
For $k >0$, let $\Mer{k} = \{\e_1, \ldots, \e_\ell\}$. Then
$n(f_{h,k}) = (2^h-s_k) + 2\e_1 (2^1) + \cdots + 2\e_\ell (2^\ell)$.
For each $i \in [\ell]$, write $2^i = M_i+1$. Then
$n(f_{h,k}) = 2^h-s_k + 2k + 2 \s (\Mer{k})$.	
By Lemma \ref{SFormula}, $s_k = 3k + 2 \s(\Mer{k})$.
Inserting this into the equation above, we find 
$n(f_{h,k}) = 2^h - k$, as was to be shown.
\end{proof}

For each positive integer $h$, let 
$F_h = \{ f_{h,k} : s_k \le 2^h \}$.
The collections $F_4$ and $F_5$
are presented in Table \ref{PebConfigs}.

\begin{table}[htb]\caption{The collections $F_4$ and $F_5$}
\begin{tabular}[t]{clc} 	\toprule
$k$		&$f_{4,k}$				&$n(f_{4,k})$			\\ \hline
0		&$\{16,0,0,0,0\}$		&16					\\ 
1		&$\{11,2,0,0,0\}$		&15					\\ 
2		&$\{6,4,0,0,0\}$		&14					\\ 
3		&$\{5,0,2,0,0\}$		&13					\\ 
4		&$\{0,2,2,0,0\}$		&12					\\ 
\bottomrule
\end{tabular}
\quad\quad
\begin{tabular}[t]{clc} 	\toprule
$k$		&$f_{5,k}$				&$n(f_{5,k})$			\\ \hline
0		&$\{32,0,0,0,0,0\}$		&32					\\ 
1		&$\{27,2,0,0,0,0\}$		&31					\\ 
2		&$\{22,4,0,0,0,0\}$		&30					\\ 
3		&$\{21,0,2,0,0,0\}$		&29					\\ 
4		&$\{16,2,2,0,0,0\}$		&28					\\ 
5		&$\{11,4,2,0,0,0\}$		&27					\\ 
6		&$\{10,0,4,0,0,0\}$		&26					\\ 
7		&$\{9,0,0,2,0,0\}$		&25					\\ 
8		&$\{4,2,0,2,0,0\}$		&24					\\ 
\bottomrule
\end{tabular}
\label{PebConfigs}
\end{table}

\begin{lem}\label{BigReduce}
Let $h$ be an integer, $h \ge 2$, and let $f_{k,h} \in F_h$.
Then
\begin{enumerate}
\item\label{Valid} $f_{h-1, r(k)} \in F_{h-1}$, and 
\item\label{Reduce} $r(f_{h,k}) = f_{h-1, r(k)}$.
\end{enumerate}
\end{lem}

\begin{proof} For reference,
\[
f_{h-1, r(k)} = \join [\{2^{h-1}-s_{r(k)}\}, 2 \pad_{h-1} (\Mer{r(k)})].
\]
By Corollary \ref{ReduceUpperBound}, $s_{r(k)} \le 2^{h-1}$,
which demonstrates that $f_{h-1, r(k)} \in F_{h-1}$,
proving part (\ref{Valid}).

To establish part (\ref{Reduce}), we begin 
with the calculation of the head of $r(f_{h,k})$; see equation (\ref{Head}).
The sum of the elements of $2 \pad_h(\Mer{k})$
is $2 \s(\Mer{k})$. Using Lemma \ref{SReductionFormula},
the head of $r(f_{h,k})$ is 
\begin{align*}
\frac{(2^h - s_k) + \s(\Mer{k})}{2} + 2 \proj_1(k) 
	&= 2^{h-1} - \left(\frac{s_k - \s(\Mer{k})}{2} - 2 \proj_1(k) \right)	\\
	&= 2^{h-1} - s_{r(k)},
\end{align*}
which shows that $r(f_{h,k})$ and $f_{h-1, r(k)}$ have
the same head. We are left to show that the remaining $h-1$ coordinates
of $r(f_{h,k})$ and $f_{h-1, r(k)}$ are equal, that is, we must show
\begin{equation}\label{Last}
S(2 \pad_h (\Mer{k})) = 2 \pad_{h-1} (\Mer{r(k)}).
\end{equation}
But $S(2 \pad_h (\Mer{k})) = 2 \pad_{h-1} (S(\Mer{k}))$ and, by
the definition of the reduction of an integer, 
$S(\Mer{k}) = \Mer{r(k)}$, which establishes equation (\ref{Last})
and draws our proof to a conclusion.
\end{proof}

We are now prepared to state and prove the main result of
this section.

\begin{thm}\label{MainUpperResult}
For each nonnegative integer $h$, 
$\pi^* (T^h) \le 2^h - k(h).$
\end{thm}

\begin{proof}
We begin by showing that each pebbling configuration
in $F_h$ pebbles $T^h$. For $h=1$, there is only one pebbling configuration; namely, $f_{1,0} = \{2,0\}$, and it is easy to see that this configuration
pebbles $T^1$. Now let $h \ge 1$ and let us suppose that each pebbling configuration in $F_h$ pebbles $T^h$. Let $f_{h+1,k} \in F_{h+1}$.
By Lemma \ref{BigReduce}, $r(f_{h+1,k}) = f_{h,r(k)} \in F_h$. 
Since $f_{h,r(k)}$ pebbles $T^h$, $f_{h+1,k}$ pebbles $T^{h+1}$.

Let $h$ be a nonnegative integer. To finish our proof, 
observe that $s_{k(h)} \le 2^h$; thus, $f_{h,k(h)}$ pebbles $T^h$.
However, by Lemma \ref{NumF}, $n ( f_{h,k(h)} ) = 2^h - k(h)$, which shows
that $\pi^* (T^h) \le 2^h - k(h).$
\end{proof}

\section{The lower bound of $\pi^* (T^h)$}
\label{lower}

Here is the main result of this section.

\begin{thm}\label{MainLowerResult}
For each positive integer $h$, 
$\pi^* (T^h) \ge 2^h - k(h).$
\end{thm}

The proof of this theorem relies on two noteworthy results of Fu and Shiue. 
According to Lemma 3.3 and Theorem 3.4
of their paper, an optimal \emph{symmetric} and \emph{even} 
pebbling configuration on $T^h$ is an optimal pebbling configuration on $T^h$.
In effect, our proof of Theorem \ref{MainLowerResult}
shows that a symmetric and even pebbling configuration 
on $T^{h-1}$ that contains $2^h - k(h) - 1$ pebbles cannot pebble $T^h$.
Hereafter, we consider only symmetric and even pebbling
configurations on $T^h$.

Our next lemma is a simple necessary condition for pebbling.

\begin{lem}\label{lem:necessary}
If $f = \{f_0, \ldots, f_h\}$ pebbles $T^h$, then 
\begin{equation}\label{eq:liquid}
3 n(f) - c(f) \ge 2^{h+1}.
\end{equation}
\end{lem}

\begin{proof} The key idea is to treat the pebbles as units of liquid, 
that is, as infinitely divisible units. In the spirit of pebbling,
if a unit of liquid (a pebble) is distance $d$ from a 
specified leaf, then it can deliver $1/2^d$ units of liquid to that leaf. 
We show that when the inequality (\ref{eq:liquid}) 
is satisfied, then $f$ can deliver a unit of liquid to a specified leaf, 
which is a necessary condition for delivering a pebble to that leaf.

Let a leaf be specified and consider the path from the root to this leaf.
We call this path the \emph{spine}. Let the rest of the tree be
called the \emph{remainder}. 
For example, the tree $T^4$, separated into a spine and remainder,
is pictured in Figure \ref{fig:SpineRemainder}.

\begin{figure}[htb]
\includegraphics[width=.7\textwidth]{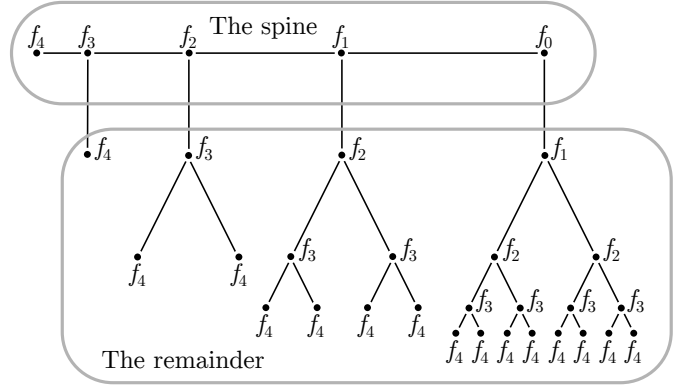}
\caption{The tree $T^4$ is separated into its spine and remainder}
\label{fig:SpineRemainder}
\end{figure}

The contribution to the leaf from the pebbles in the remainder at distance $h$ 
from the root is
\[
\frac{2^{h-1} f_h}{2^{2h}} + \frac{2^{h-2} f_h}{2^{2h-2}} + \cdots
+ \frac{2^1 f_h}{2^4} + \frac{f_h}{2^2} 
= \frac{f_h M_h}{2^{h+1}}.
\]
In general, for $d \in \{1, \ldots, h\}$,
the contribution to the leaf from the remainder from the 
pebbles at distance $d$ from the root is $f_d M_d/2^{h+1}$. 
In total, the contribution to the leaf from the remainder is
\[
\frac{1}{2^{h+1}} \left( f_1 M_1 + \ldots + f_h M_h \right)
= \frac{1}{2^{h+1}} \left( n(f) - c (f) \right).
\]
The contribution to the leaf from the spine is
$2 n(f)/2^{h+1}.$
Consequently, the total contribution to the leaf from the (liquid) pebbles
in the configuration $f$ is $(3 n (f)  - c (f))/2^{h+1}.$
If $f$ can deliver at least one pebble to the leaf, 
then $f$ can deliver at least one unit of liquid to the leaf, 
which completes our proof.
\end{proof}

The following sequence plays an important role in the proof
of Theorem \ref{MainLowerResult}. Let $c_1 = 2$ and, for $k \ge 2$,
let $c_k = 3(2^k)-2$. An important feature of this 
sequence is contained in the next lemma, which we state without 
proof.

\begin{lem}\label{subadditive} For $i \ge 1$,
$3c_1 + c_2 + c_3 + \cdots + c_i = c_{i+1} - (i+1)2.$
\end{lem}

Let $h$ and $m$ be positive integers with $m < 2^{h+2}$. 
The binary expansion of $m$ is 
$
m = \d_0 + 2^1 \d_1 + 2^2 \d_2 + \cdots + \d_{h+1} 2^{h+1}, 
$
where $\d_i \in \{0,1\}$ for each integer $i$, $0 \le i \le h+1$.
The configuration
\[
\psi_{h,m} = \{ \d_0 + 2 \d_1, 2 \d_2, \ldots, 2\d_{h+1}\}
\]
is an even, symmetric pebbling configuration on $T^h$ containing $m$
pebbles. It is easy to show that $\psi_{h,m}$ maximizes
$3 n(f) - c(f)$ among the set of even, symmetric pebbling configurations 
$f$ on $T^h$ containing $m$ pebbles. 
Let $t_m = 3 n(\psi_{h,m} ) - c ( \psi_{h,m} )$; it can be shown that
\begin{equation}\label{Tsequence}
t_m = (\d_0 + \d_1 2^1)c_1  + \d_2 c_2 + \cdots + \d_{h+1} c_{h+1}.
\end{equation}

Finally, here is the proof of Theorem \ref{MainLowerResult},
the main result of this section.

\begin{proof}[Proof of Theorem \ref{MainLowerResult}]
The result is true for $h=1$ by inspection.
Let $h \ge 2$ be given. 
We show that $t_{2^h-k(h)-1} < 2^{h+1}$, which, 
according to Lemma \ref{lem:necessary}, proves the theorem.

Let $2^h = r + \e_1 \mu_1 + \cdots + \e_\ell \mu_\ell$ be
the $\mu$-expansion of $2^h$. Then $k(h) = \e_1 M_1 + \cdots + \e_\ell M_\ell$
and $s_{k(h)} = \e_1 \mu_1 + \cdots + \e_\ell \mu_\ell$. Thus
\begin{align*}
2^h - k(h) 
	&= r + \e_1 (\mu_1 - M_1) + \cdots + \e_\ell (\mu_\ell - M_\ell)	\\
	&= r + \e_1 2^2 + \cdots + \e_\ell 2^{\ell+1}
\end{align*}
Our proof is divided into cases, depending on whether or not
$r=0$.

First, let us suppose that $r=0$ and that $\e_i \ne 2$
for each $i \in [\ell]$ in the $\mu$-expansion of $2^h$. 
Then the binary expansion of $2^h - k(h)$ is 
$\e_1 2^2 + \cdots + \e_\ell 2^{\ell+1}$ and therefore
\begin{align*}
t_{2^h-{k(h)}} 
	&= \e_1 c_2 + \e_2 c_3 + \cdots + \e_{\ell} c_{\ell+1}	\\
	&= 2(\e_1 \mu_1 + \e_2 \mu_2 + \cdots + \e_{\ell} \mu_{\ell})	\\
	&= 2 s_{k(h)} 	\\
	&= 2^{h+1}.
\end{align*}
Since the sequence $\{t_k\}$ is strictly increasing, $t_{2^h-{k(h)}-1} < 2^{h+1}$.

Next, let us suppose that $r=0$ but that $\e_j = 2$ for some 
$j \in [\ell]$ in the $\mu$-expansion of $2^h$. In particular,
this implies that $\e_i \in \{0,1\}$ for each integer 
$i \in \{j+1, \ldots, \ell\}$. Then 
\[
2^h - {k(h)} = 2(2^{j+1}) + \e_{j+1} 2^{j+2} + \cdots \e_{\ell} 2^{\ell+1}.
\]
Therefore, the binary expansion of $2^h - k(h)-1$ is
\[
1 + 2 + 2^2 +  \cdots + 2^j + 2^{j+1} + \e_{j+1} 2^{j+2} + \cdots + \e_{\ell} 2^{\ell+1}
\]
hence
\[
t_{2^h-{k(h)}-1} 
	= 3 c_1 + c_2 + \cdots + c_j +c_{j+1} + 
		\e_{j+1} c_{j+2} + \cdots + \e_{\ell} c_{\ell+1}.
\]
By Lemma \ref{subadditive}, we may conclude 
\begin{align*}
t_{2^h-{k(h)}-1} 
	&= 2 c_{j+1} + \e_{j+1} c_{j+2} + \cdots \e_{\ell} c_{\ell+1} -(j+1)2	\\
	&= 2 (2 \mu_{j} + \e_{j+1} \mu_{j+1} + \cdots \e_{\ell} \mu_{\ell}) -(j+1)2	\\
	&= 2 s_{k(h)} -(j+1)2 \\
	&= 2^{h+1} - (j+1)2.
\end{align*}
In particular, this shows that $t_{2^h-{k(h)}-1} < 2^{h+1}$,
as was to be shown.

Lastly, let us assume that $r \in \{1,2,3,4\}$ in the $\mu$-expansion
of $2^h$. This implies that $\e_i \in \{0,1\}$ for each $i \in [\ell]$.
Then $r-1 \in \{0,1,2,3\}$ and the binary expansion of $2^h -{k(h)} - 1$ 
is 
\[
(r-1) + \e_1 2^2 + \e_2 2^3 + \cdots + \e_\ell 2^{\ell+1}.
\]
Thus,
\begin{align*}
t_{2^h -{k(h)} - 1} 
	&= 2(r-1) + \e_1 c_2 + \e_2 c_3 + \cdots + \e_\ell c_{\ell+1}		\\
	&= 2(r-1) + \e_1 2\mu_1 + \e_2 2\mu_3 + \cdots + \e_\ell 2\mu_{\ell}		\\
	&= 2 \left( (r-1) + \e_1 \mu_1 + \e_2 \mu_3 + \cdots + \e_\ell \mu_{\ell} \right)	\\
	&= 2^{h+1} - 2,
\end{align*}
as was to be shown.
\end{proof}

\section{Asymptotic analysis of $k(h)$} 
\label{asymptotics}

In this section we study the asymptotic behavior of $k(h)$
as $h \to \infty$.
Given a positive integer $m$, let $s^{-1} (m) = \max \{n : s_n \le m \}.$
In this notation, $k(h) = s^{-1} (2^h).$ Our first step is to
develop a formula for the inverse of $s$.

Let $m$ be a positive integer with $\mu$-expansion
$m = r + \e_1 \mu_1 + \cdots + \e_\ell \mu_\ell.$
Let $\phi(m) = \e_1 + \cdots + \e_\ell$.
For example, the $\mu$-expansions of 236 and 253 are $2\mu_2 + \mu_3 + \mu_6$ and 
$4 + \mu_2 + \mu_4 + \mu_6$; thus $\phi(236) = 4$ and $\phi(253) = 3$.

\begin{thm}\label{sInverse} Given a positive integer $m$,
there exists an integer $r \in \{0,1,2,3,4\}$ such that
$m - r = 3s^{-1} (m) + 2 \phi (m)$.
\end{thm}

\begin{proof}
Let $m = r + \e_1 \mu_1 + \cdots + \e_\ell \mu_\ell$ be the
$\mu$-expansion of $m$. Since $\mu_i = 3 M_i + 2$ for each 
positive integer $i$, 
\[
m -  r = 3 ( \e_1 M_1 + \cdots + \e_\ell M_\ell) + 2 (\e_1 + \ldots + \e_\ell).
\]
The right-hand side is $3 s^{-1} (m) + 2 \phi (m)$,
as was to be shown.
\end{proof}

Let $h$ be a positive integer. Since $k(h) = s^{-1} (h)$,
Theorem \ref{sInverse} asserts that 
there exists an integer $r \in \{0,1,2,3,4\}$ such that
\begin{equation}\label{Asymptotics}
k(h) = 2^h/3 - 2\phi(2^h)/3 -r/3.
\end{equation}
Therefore, in order to understand the asymptotic behavior of $k(h)$, 
we need to investigate the asymptotic behavior of $\phi(2^h)$. To this end, let
\begin{equation}\label{AlphaSequence}
\a (h) = \frac{2 \phi(2^h) - (h+2)}{\log_2 (h+2)}.
\end{equation}
We prove the following theorem.

\begin{thm} For $h \ge 2$, $-1 \le \a (h) \le 1.$ 
In addition, $\liminf_{h \to \infty} \alpha (h) = -1$ and 
$\limsup_{h \to \infty} \alpha (h) = 1.$
\end{thm}

\begin{proof} Let $h$ be an integer, $h \ge 2$. 
Let $x^* = x^*(h)$ be the root of the equation
\begin{equation}\label{Root}
h - 2 x - 1 - \log_2 (x + 1) =0,
\end{equation}
and let $j^* = \lceil x^* \rceil.$ Then 
$j^* = \min \{ j : j \ge 2^{h-2j-1} - 1 \}.$

First we develop the $\mu$-expansion of $2^h$ and
show that $\phi (2^h) \ge j^* +1$. To get started,
notice that $\mu_{h-2} \le 2^h < \mu_{h-1}$ and that
$2^h - \mu_{h-2} = 2^{h-2} + 1$. We can continue this process
of successive subtractions $j^*$ times, obtaining 
\[
2^h - \sum_{i=1}^{j^*} \mu_{h-2i} = 2^{h-2j^*} + j^*.
\]
At this point, the simple pattern of subtractions is disrupted;
the next subtraction is $\mu_{h-2j^*-1}$, which yields
\[
2^h - \mu_{h-2} - \cdots - \mu_{h-2j^*} - \mu_{h-2j^*-1} = A,
\]
where $A = j^*- \left(2^{h-2j^*-1} - 1\right).$ 
Thus $\phi(2^h) = j^* + 1 + \phi(A)$.

An analysis of equation (\ref{Root}) reveals that 
\[
j^*+ 1 \ge \frac{1}{2} (h+2) -\frac{\log_2(h+2)}{2}.
\]   
Likewise,
$A \le 2^{h-2j^*} + 2^{h-2j^*-1}$ hence $0 \le \phi(A) \le h-2j^*-1$ 
and we may conclude that
\[
(h+2) - \log_2 (h+2) \le 2 \phi (2^h) \le (h+2) + \log_2 (h+2),
\]
or $-1 \le \a (h) \le 1$.

We are left to prove the claims about the limits inferior and
superior of $\a(h)$. We consider two families
of trees.

\begin{enumerate}
\item\label{ThinTree} Let $h=h_k= 2^{k+1}+k-1$. Then $j^* = j^*_k = 2^k-1$ and
\[
2^h - \sum_{i=1}^{j^*} \mu_{h-2i} = \mu_{h-2j^*-1}.
\]
Accordingly, $\phi (2^h) = j^* + 1 = 2^k$, and
$\lim_{k \to \infty} \a (h_k) = -1,
$
which shows $\liminf_{h \to \infty} \a (h) = -1$.
\item\label{ThickTree} Let $h=h_k= 2^{k+2}-k$. Then $j^* = j^*_k = 2^{k+1}-k$
and 
\begin{align*}
2^h - \sum_{i=1}^{j^*} \mu_{h-2i}
	&=  2^k + (2^{k+1}-k)		\\
	&=  \mu_{k-1} + \mu_{k-2} + \cdots + \mu_2 + 2\mu_1.
\end{align*}
Thus $\phi (2^h) = j^* + k = 2^{k+1}$
and $\lim_{k \to \infty} \a (h_k) = 1$, which shows that
$\limsup_{h \to \infty} \a (h) = 1$.
\end{enumerate}
Our proof is complete.
\end{proof}

By combining equations (\ref{Asymptotics}) and (\ref{AlphaSequence}),
we find that 
\[
k(h) = \frac{1}{3}(2^h) - \frac{1}{3}(h+2) 
	- \frac{1}{3} \alpha (h) \log_2 (h+2) + O(1).
\]

The families of trees presented at the end of the proof
reveal some interesting examples of optimal pebblings.
In case (a), the first level that contains any pebbles is $k$
and thereafter, the pebbling configuration follows a regular,
alternating pattern. In other words, the top of the tree is
empty. In case (b), each of the levels 1 through $k-1$ contain pebbles
and thereafter the pebbling configuration follows a regular
alternating pattern.
In other words, the top of the tree, except for the root itself, is full.

Let levels 0 through $h-2j^*-2$ designate the \emph{top} of $T^h$
and let the remaining levels be called the \emph{bottom} of the
tree $T^h$. We can see that the bottom of the tree $T^h$ 
starts at approximately level 
\[
h- 2j^*-1 \sim \log_2(h/2+1).
\]
Our analysis reveals that an optimal pebbling of $T^h$ is
regular on the bottom: there are 2 or 4 pebbles at level $h-2j^*-1$
and 2 pebbles in each of the levels $h-2k$, $k \in \{1, \ldots, j^*\}$.
The top of $T^h$, however, may vary from full to empty.

\section{Connections with the Connolly-Fox sequence}
\label{metafib}

Recall from \S\ref{intro} that the Conolly-Fox sequence $\{c_n\}$ satisfies
the recurrence relation (\ref{ConollyFox})
with initial conditions $c(1)=1$ and $c(2)=2$.
We will prove the following theorem.

\begin{thm}\label{StoC} For each positive integer
$n$, $s_n = 4 c_n + n$.
\end{thm}

\begin{proof}
We begin by defining a list of numbers composed entirely of 0s and 1s.
Let $D_1 = \{1\}$. We define successive lists recursively:
for each integer $k$, $k \ge 2$, let
$D_k = \join[D_{k-1}, D_{k-1}, \{0\}].$
The lists $D_1$ through $D_4$ are collected in Table \ref{DSequence}.

\begin{table}[htb]
\caption{The lists $D_1$ through $D_4$}
\begin{tabular}{r l}\toprule
$k$	& $D_k$	\\	\hline
1	&\{1\} 	\\
2	&\{1, 1, 0\} 	\\
3	&\{1, 1, 0, 1, 1, 0, 0\} \\
4	&\{1, 1, 0, 1, 1, 0, 0, 1, 1, 0, 1, 1, 0, 0, 0\} \\
\bottomrule
\end{tabular}
\label{DSequence}
\end{table}

Let $D$ be the limit of this sequence of lists and, for each positive integer $n$, let $d_n$ denote the $n$th element of $D$. 
The sequence $\{d_n\}$ is the sequence of differences in the 
Conolly-Fox sequence; see \href{https://oeis.org/A079559}{OEIS, A079559}
\cite{OEIS}.
Thus, for each positive integer $n$, $c_n = d_1 + \cdots+ d_n$.
The initial terms of the sequences $\{d_n\}$ and $\{c_n\}$
are presented in Table \ref{PartialSums2}. 

\begin{table}[htb]
\caption{The initial terms of $\{d_n\}$ and $\{c_n\}$}
\begin{tabular}{rrrrrrrrrrrrrrrr} \toprule
$n$	&1 & 2 & 3 & 4 & 5 & 6 & 7 & 8 & 9 &
   10 & 11 & 12 & 13 & 14 & 15 \\
$d_n$	&1 & 1 & 0 & 1 & 1 & 0 & 0 & 1 & 1 &
   0 & 1 & 1 & 0 & 0 & 0 \\
$c_n$ 	&1 & 2 & 2 & 3 & 4 & 4 & 4 & 5 & 6 &
   6 & 7 & 8 & 8 & 8 & 8 \\
\bottomrule
\end{tabular} 
\label{PartialSums2}
\end{table}

Recall the sequence $\{a_n\}$ defined in \S1; see Table \ref{PartialSums}.
For each positive integer $n$, it is easy to see that 
$a_n = 4 d_n +1$ and therefore $s_n = 4 c_n + n$, as was to be shown.
\end{proof}

\bibliography{TreePebbles}        

\end{document}